\documentclass[letterpaper,11pt]{amsart}
\usepackage{amsmath,amssymb,amsthm,pinlabel,tikz}
\usepackage{verbatim}
\newcommand{\nc}{\newcommand}
\nc{\dmo}{\DeclareMathOperator}
\dmo{\ra}{\rightarrow}
\dmo{\N}{\mathbb{N}}
\dmo{\Z}{\mathbb{Z}}
\dmo{\R}{\mathbb{R}}
\dmo{\C}{\mathcal{C}}
\dmo{\AC}{\mathcal{AC}}
\dmo{\Mod}{Mod}
\dmo{\PMod}{PMod}
\dmo{\B}{B}
\dmo{\PB}{PB}
\dmo{\I}{\mathcal{I}}
\dmo{\el}{\ell_{\C}}
\dmo{\NN}{\mathcal{N}}

\nc{\nt}{\newtheorem}

\newtheorem{thm}{{\bf Theorem}}[section]
\newtheorem{lem}[thm]{{\bf Lemma}}

\newtheorem{prop}[thm]{{\bf Proposition}}

\numberwithin{equation}{section}

\title[Translation lengths of Torelli groups and pure braid groups]{Minimal asymptotic translation lengths of Torelli groups and pure braid groups on the curve graph}

\author[H. Baik]{Hyungryul Baik}
\address{%
		Department of Mathematical Sciences, KAIST,
		291 Daehak-ro Yuseong-gu, Daejeon, 34141, South Korea 
}
\email{%
        hrbaik@kaist.ac.kr
}

\author[H. Shin]{Hyunshik Shin}

\address{%
        Department of Mathematical Sciences, KAIST,
		291 Daehak-ro Yuseong-gu, Daejeon, 34141, South Korea 
}
\email{%
        hshin@kaist.ac.kr
}

\begin{document}
\begin{abstract}
In this paper, we show that the minimal asymptotic translation length of 
the Torelli group $\I_g$ of the surface $S_g$ of genus $g$
on the curve graph asymptotically behaves like $1/g$,
contrary to the mapping class group $\Mod(S_g)$, which behaves like $1/g^2$.
We also show that the minimal asymptotic translation length of 
the pure braid group $\PB_n$ on the curve graph asymptotically behaves like $1/n$,
contrary to the braid group $\B_n$, which behaves like $1/n^2$. 
\end{abstract}

\maketitle

%%%%%%%%%%%%%%%%%%%%%%%%%%%%%%%%%%%%%%%%%%%%%%%%%%%%%%%%%%%%%%%%%%%%%%
%
%	Introduction
%
%%%%%%%%%%%%%%%%%%%%%%%%%%%%%%%%%%%%%%%%%%%%%%%%%%%%%%%%%%%%%%%%%%%%%%

\section{Introduction}	\label{section:introduction}
Let $S=S_{g,n}$ be an orientable surface of genus $g$ with $n$ punctures.
The \textit{mapping class group} of $S$, denoted by $\Mod(S)$,
is the group of isotopy classes of orientation-preserving homeomorphisms of $S$.
%that fix the boundary pointwise.
An element of $\Mod(S)$ is called a mapping class.
The curve graph $\C(S)$ of $S$ is one of the important spaces upon which 
$\Mod(S)$ is acting. The \textit{curve graph} is a simplicial graph whose vertices are 
isotopy classes of essential simple closed curves in $S$ and two vertices are joined by
an edge if they are realized by a pair of disjoint curves.
Assigning each edge length 1 induces a path metric $d_{\C}(\cdot,\cdot)$ on $\C(S)$.
Then $\Mod(S)$ acts on $\C(S)$ by isometries.
For an element $f \in \Mod(S)$,
we define the \textit{asymptotic translation length} (also known as
\textit{stable translation length}) of $f$ by
$$ \ell_{\C} (f) = \liminf_{j \ra \infty} \frac{d_{\C}(\alpha, f^j(\alpha))}{j},$$
where $\alpha$ is an element in $\C(S)$. Note that $\ell_{C}(f)$ is independent of
the choice of $\alpha$,
and $\ell_{C}(f^m) = m \ell_{C}(f)$ for each $m \in \N$.

In this paper, We always assume that the complexity of $S$ is $\xi(S) = 3g-3+n \geq 2$.
Masur and Minsky \cite{MasurMinsky99} showed that $C(S)$ is Gromov-hyperbolic 
and show that $\el(f)>0$ if and only if $f$ is a pseudo-Anosov mapping class.
Bowditch \cite{Bowditch08} showed that there exists some positive integer $m$,
depending only on $S$, such that for any pseudo-Anosov $f \in \Mod(S)$,
$f^m$ acts by translation along a geodesic axis in $\C(S)$.
As a consequence, $\el(f)$ is a rational number with bounded
denominator. There has been work of many authors on estimating
asymptotic translation lengths on curve graphs. For instance, see 
\cite{FarbLeiningerMargalit08, GadreTsai11, GadreHironakaKentLeininger13, Valdivia14, AougabTaylor15, KinShin17, Valdivia17,
  BaikShinWu18} and references therein. 

One can think of the behavior of the minimal asymptotic translation length on a closed
surface $S_g$ of genus $g$. For any subgroup $H < \Mod(S)$, let us define
$$L_{\C}(H) = \min \{ \el(f) | f \in H, \textrm{pseudo-Anosov} \}.$$
Gadre and Tsai \cite{GadreTsai11} proved that
$$L_{\C}(\Mod(S_g)) \asymp \frac{1}{g^2},$$
where $F(g) \asymp G(g)$ implies that there exist positive constants $C$ and $D$ 
such that $C G(g) \leq F(g) \leq D G(g)$.
The second author and Kin \cite{KinShin17} showed that 
the minimal asymptotic translation lengths of
hyperelliptic mapping class group, handlebody group, and hyperelliptic handlebody group
of $S_g$ also behave like $1/g^2$.

The Torelli group $\I_g$ is another important subgroup of $\Mod(S_g)$.
Farb--Leininger--Margalit \cite{FarbLeiningerMargalit08} proved that 
$L_{\C}(\I_g) \ra 0$ as $g \ra \infty$ but the exact asymptote is not known.
It turns out that the behavior of $L_{\C}(\I_g)$ is different
from that of $L_{\C}(\Mod(S_g))$.

\begin{thm}	\label{thm:Torelli}
For $g \geq 2$, we have $$L_{\C}(\I_g) \asymp \frac{1}{g}.$$
\end{thm}
\noindent
The lower bound for Theorem \ref{thm:Torelli} is a direct consequence of the result by Tsai
\cite[Lemma 3.1]{Tsai09} together with Bestvina--Handel
algorithm \cite{BestvinaHandel95}.
For the upper bound of $L_{\C}(\I_g)$, we use an explicit sequence of pseudo-Anosov 
$f_g \in \I_g$ such that $\el(f_g) \leq D/g$ for some constant $D>0$.
This sequence is arising from Penner's construction \cite{Penner88} 
and this technique can be applied to all other surfaces.

We also investigate the minimal asymptotic translation lengths
of the pure braid group $\PB_n$.
The braid group $\B_n$ can be regarded as the mapping class group of the $n$-punctured disk $D_n$
fixing boundary pointwise. Then the pure braid group $\PB_n$ is analogous to the Torelli group
$\I_g$ in the sense that the action of $\PB_n$ on the first homology $H_1(D_n)$ is trivial.
(For an introduction to braid groups and pure braid groups, we refer the reader 
to \cite[Chapter 9]{FarbMargalit12} or \cite{BirmanBrendle05}.)
Our next theorem in this paper is as follows.
\begin{thm}\label{thm:PBn}
For a pure braid group $\PB_n$, we have
$$L_{\C}(\PB_n) \asymp \frac{1}{n}.$$
\end{thm}
\noindent
On the contrary, it is known that $L_{\C}(\B_n) \asymp 1/n^2$ (see \cite[Theorem B]{KinShin17}).
To obtain the lower bound, we cannot use the Lemma 3.1 of \cite{Tsai09}
as in the Torelli group.
Instead, we use the calculations in \cite{GadreTsai11} and the nesting lemma by Masur--Minsky \cite{MasurMinsky99}.
The upper bound of $L_{\C}(\PB_n)$ is again obtained by an explicit sequence of pure braids 
$f_n \in \PB_n$ from Penner's construction.

We also discuss the minimal translation length for the pure mapping class
group $\PMod(S)$. In fact, our method in the proof of Theorem \ref{thm:PBn} can
be easily adapted in this case, and one can obtain another proof of
the following theorem of Valdivia.
\begin{thm}[ c.f., \cite{Valdivia14} ] \label{thm:PMod} 
For any fixed $g \geq 0$, 
$$L_{\C}(\PMod(S_{g,n})) \asymp \frac{1}{n}.$$
\end{thm}  

The original theorem by Valdivia in \cite{Valdivia14}
was proved for $g \geq 2$ using a technique which is similar to one we used to
prove Theorem \ref{thm:Torelli}. 
Our proof also works for $g=$ 0 or 1 as well as $g \geq 2$,
and hence we can state this theorem for all $g\geq 0$.

\subsection*{Acknowledgements}
We thank Dan Margalit for suggesting this problem and helpful
comments. We thank Chenxi Wu for fruitful discussions. 
The authors also thank the anonymous referees for their valuable comments which improved the paper very much. 
The first author was partially supported by Samsung Science \& Technology Foundation grant No. SSTF-BA1702-01.
The second author was supported by Basic Science Research Program
through the National Research Foundation of Korea(NRF) funded by the Ministry of Education
(NRF-2017R1D1A1B03035017).

%%%%%%%%%%%%%%%%%%%%%%%%%%%%%%%%%%%%%%%%%%%%%%%%%%%%%%%%%%%%%%%%%%%%%%%%%%%%%%%
%
%	Background
%
%%%%%%%%%%%%%%%%%%%%%%%%%%%%%%%%%%%%%%%%%%%%%%%%%%%%%%%%%%%%%%%%%%%%%%%%%%%%%%%
\medskip
\section{Background}
\subsection{Train tracks and Bestvina--Handel algorithm}	\label{section:traintrack}
In the late 1970's, Thurston introduced a powerful tool to study measured geodesic
laminations on a surface, so called train tracks. 
For more discussion about train tracks, see \cite{PennerHarer92} or \cite{BestvinaHandel95}.

A \textit{train track} $\tau$ is a smooth 1-complex embedded in a surface $S$ where 
there is a well-defined tangent line at each vertex, and that there are edges tangent in each direction.
Vertices and edges of $\tau$ are also called switches and branches, respectively.
We require that each vertex of $\tau$ is at least tri-valent.
We assign each edge of $\tau$ a nonnegative number, called a \textit{weight},
so that it satisfies switch conditions. That is, the sum of weights on incoming edges
is equal to the sum of weights on outgoing edges.
The set $\mu$ of weights on $\tau$ is called a \textit{measure}.

Every pseudo-Anosov $f \in \Mod(S)$ has a weighted train track $\tau$ such that
$f(\tau)$ collapses to $\tau$. We call such $\tau$ an invariant train track for $f$.
In \cite{BestvinaHandel95}, Bestvina and Handel  gave an algorithm to find an invariant
train track $\tau$ for a given pseudo-Anosov $f$.
The branches of $\tau$ consist of two types, \textit{real} branches and \textit{infinitesimal}
branches.
There are at most $9|\chi(S)|$ real branches when $S$ is closed, and
at most $3|\chi(S)|$ real branches when $S$ is punctured. Furthermore,
there are at most $24|\chi(S)|-8n$ infinitesimal branches (see Section 4 of \cite{GadreTsai11}).

Let $\mathcal{R}$ be the set of real branches of $\tau$.
In \cite{BestvinaHandel95}, Bestvina and Handel showed that the transition matrix $M$ of $f$ on $\tau$ is of the form
$$M=
\begin{pmatrix}
A & B \\
O & M_{\mathcal{R}}
\end{pmatrix},
$$
where $A$ is a permutation matrix on infinitesimal branches,
and $M_{\mathcal{R}}$ is the transition matrix on real branches where
$M_{\mathcal{R}}$ is \textit{primitive}, that is,
for some $m \in \N$, $M_{\mathcal{R}}^m$ is a positive matrix .

\medskip
\subsection{Nesting lemma and lower bound of asymptotic translation length}
Let $\tau$ be a train track,
and let $P(\tau)$ be the polyhedron of measures supported on $\tau$, that is, 
a cone satisfying switch conditions in $\R^B_{\geq 0}$, where $B$ is the number of
branches of $\tau$.
We say that a measure $\mu$ on $\tau$ is positive if it has positive weights on every branch.
Let $int(P(\tau))$ be the set of positive measures supported on $\tau$.
We say a train track $\tau$ is \textit{recurrent} if there is a positive measure $\mu \in int(P(\tau))$.
A train track is \textit{transversely recurrent} if given a branch of $\tau$,
there is a simple closed curve on $S$ that crosses the branch, intersects $\tau$
transversely, and the union of $\tau$ and the simple closed curve has no complementary
bigons. A train track is said to be \textit{birecurrent} if it is both recurrent and transversely recurrent.

We say that a train track $\tau$ fills the surface $S$ if the complementary region $S \setminus \tau$
is a union of ideal polygons containing at most one puncture.
If $\tau$ fills $S$, then a train track $\sigma$ is called a 
\textit{diagonal extension} of $\tau$ if $\tau$ is sub-track of $\sigma$ and
each branch of $\sigma \setminus \tau$ has its endpoints terminating in the cusps 
of complementary regions of $\tau$. 
Let 
$$PE(\tau) = \bigcup_{\sigma \in E(\tau)} P(\sigma),$$
and let $int(PE(\tau))$ be the set of measures in $PE(\tau)$ that are positive on each
branch of some diagonal extension of $\tau$. 
Abusing the notation, $int(PE(\tau))$ and $PE(\tau)$ also denote the set of curves which defines measures in $int(PE(\tau))$ and $PE(\tau)$, respectively. 
Masur and Minsky showed the following lemma in \cite{MasurMinsky99}. 
\begin{lem}[Nesting Lemma]
Let $\tau$ be a birecurrent train track that fills the surface $S$. Then
$$\mathcal{N}_1(int(PE(\tau))) \subset PE(\tau),$$
where $\mathcal{N}_1(X)$ is the 1-neighborhood of $X$ in the curve complex.
\end{lem}
From this lemma, Gadre and Tsai established the way to obtain lower bound of asymptotic
translation length on the curve graph of a given pseudo-Anosov mapping class.
Following the proof of Lemma 4.3 and Theorem 5.1 in \cite{GadreTsai11}, 
we have the following result (see also \cite[Proposition 3.6]{GadreHironakaKentLeininger13}).
%\begin{prop} \label{prop:lowerbound}
%If $\tau$ is a maximal birecurrent invariant train track for a pseudo-Anosov
%$f:S \ra S$ and $w \in \N$ such that $f^w(P_{\tau}) \subset int(P_{\tau})$, then
%$$\el(f) \geq \frac{1}{w}.$$
%\end{prop}

%To use this theory properly, one needs to understand the Bestvina--Handel algorithm given in \cite{BestvinaHandel95} to construct an invariant train track for pseudo-Anosov homeomorphisms and use its properties. All we need is summarized in Section 4 of \cite{GadreTsai11}. We refer the readers to \cite{GadreTsai11} for definitions regarding Bestvina--Handel algorithm. 
%Maybe we recall the definition of real and infinitestimal branches? 

%Combining the nesting lemma and Bestvina--Handel algorithm, 
%\cite{GadreTsai11} showed the following.
\begin{prop} \label{prop:lowerboundcriterion} 
Let $f \in \Mod(S_{g,n})$ be a pseudo-Anosov element and let $\tau$ be its invariant train track obtained by Bestvina--Handel algorithm. 
Let $r$ be the number of real branches and 
let $q$ be the integer such that $M_{\mathcal{R}}^q$ has a positive diagonal entry,
where $M_{\mathcal{R}}$ is the transition matrix of real branches.
If we set $k= 2qr + 24|\chi(S_{g,n})|-8n$, then for any real branch $\beta$ of $\tau$,
the path $f^k(\beta)$ traverses every branch of $\tau$.
Furthermore, combining with the nesting lemma, if we set $w := k + 6|\chi(S_{g,n})|-2n$,
then we have
$$\el(f) \geq \frac{1}{w}.$$
\end{prop}

%%%%%%%%%%%%%%%%%%%%%%%%%%%%%%%%%%%%%%%%%%%%%%%%%%%%%%%%%%%%%%%%%%%%%%%%%%%%%%%
%
%	Lower bound of Torelli groups
%
%%%%%%%%%%%%%%%%%%%%%%%%%%%%%%%%%%%%%%%%%%%%%%%%%%%%%%%%%%%%%%%%%%%%%%%%%%%%%%%
\medskip
\section{Lower bound for Torelli groups}	\label{section:Torelli_lower}

Let $X$ be a compact oriented manifold and 
let $\phi:X \ra X$ be a continuous map.
We define the graph of $\phi$ by
$\textrm{graph}(\phi) = \{ (x,\phi(x)) \in X \times X | \ x \in X \}.$
Let $\Delta$ be the diagonal of $X \times X$, that is,
$\Delta = \{ (x,x) | \ x \in X\} \subset X \times X.$
The \textit{Lefschetz number} of $\phi$, denoted by $L(\phi)$,
is defined by the algebraic intersection number 
$\hat{i}(\Delta, \textrm{graph}(\phi))$.
Since the Lefschetz number is a homotopy invariant, it is well-defined for
the homotopy class $f$ of $\phi$.
It can be computed by
$$L(f) = \sum_{i \geq 0} (-1)^i Tr(f_{\ast}^{(i)}),$$
where $f_{\ast}^{(i)}$ is the action on $H_i(X;\R)$ induced by $f$.
For complete discussion, see \cite{GuilleminPollack74} or 
\cite{BottTu82}.

In \cite{Tsai09}, Tsai gives the following crucial lemma.

\begin{lem}	\label{lem:Markovpartition}
For any pseudo-Anosov element $f \in \Mod(S_{g,n})$ equipped with
a Markov partition, if $L(f) < 0$,
then there is a rectangle $R$ of the Markov partition such that
the interior of $f(R)$ and $R$ intersect, i.e.,
the transition matrix of the Markov partition has a positive
diagonal entry.
\end{lem}

%\begin{lem}
%\label{lem:positivematrix}
%Let $M$ be a $n \times n$ integral primitive matrix.
%If there is a nonzero entry on the diagonal of $M$,
%then $M^{2n}$ is a positive matrix.
%\end{lem}

Now we are ready to obtain the lower bound for the
asymptotic translation length of a pseudo-Anosov element $f \in \I_g$.

\begin{thm}	\label{thm:lowerTorelli}
Let $g \geq 2$ and let $f$ be a pseudo-Anosov mapping class in the Torelli group $\I_g$. Then we have
$$\el(f) \geq \frac{1}{96g-96}.$$
\end{thm}

\begin{proof}
For a closed surface $S_g$ of genes $g \geq 2$, 
$H_i(S_g; \R) = 0$ for all $i \geq 3$.
Note that $H_0(S_g; \R) \simeq H_2(S_g; \R) \simeq \R$ and 
$f_{\ast}^{(i)}$ is the identity for $i = 0$ or 2.
Since $f$ is in the Torelli group, $f_{\ast}^{(1)}$ is also
the identity on $H_1(S_g; \R) \simeq \R^{2g}$.
Therefore we have the Lefschetz number $L(f) = 1 - 2g +1 = 2 -2g <0$ 
for all $g \geq 2$.
By Lemma \ref{lem:Markovpartition}, there is a positive
diagonal entry in the Markov partition matrix $M_{\mathcal{R}}$ of $f$,
that is, the transition matrix of real branches of the train track.
Hence in Proposition \ref{prop:lowerboundcriterion}, we have $q = 1$.
Since the number of real branches satisfies $r \leq 9|\chi(S_g)|$,
we have $w \leq 48 |\chi(S_g)|$ and 
$$\el(f) \geq \frac{1}{96g-96}.$$
\end{proof}

%%%%%%%%%%%%%%%%%%%%%%%%%%%%%%%%%%%%%%%%%%%%%%%%%%%%%%%%%%%%%%%%%%%%%%%%%%%%%%%
%
%	Upper bound of Torelli groups
%
%%%%%%%%%%%%%%%%%%%%%%%%%%%%%%%%%%%%%%%%%%%%%%%%%%%%%%%%%%%%%%%%%%%%%%%%%%%%%%%
\medskip
\section{Upper bound for Torelli groups}	\label{section:Torelli_upper}

In this section, we give an upper bound for $L_C(\I_g)$ 
using an explicit family of pseudo-Anosov elements $f_g$ in $\I_g$.
For a simple closed curve $a$, let $T_a$ be the left-handed Dehn twist about $a$.
We apply elements of the mapping class group from right to left.

%%%		  Updated from Here!!!!!			%%%%

\begin{thm}	\label{thm:upperTorelli}
For all $g \geq 13$,
$$L_C(\I_g) \leq \frac{8}{g-12}.$$
\end{thm}

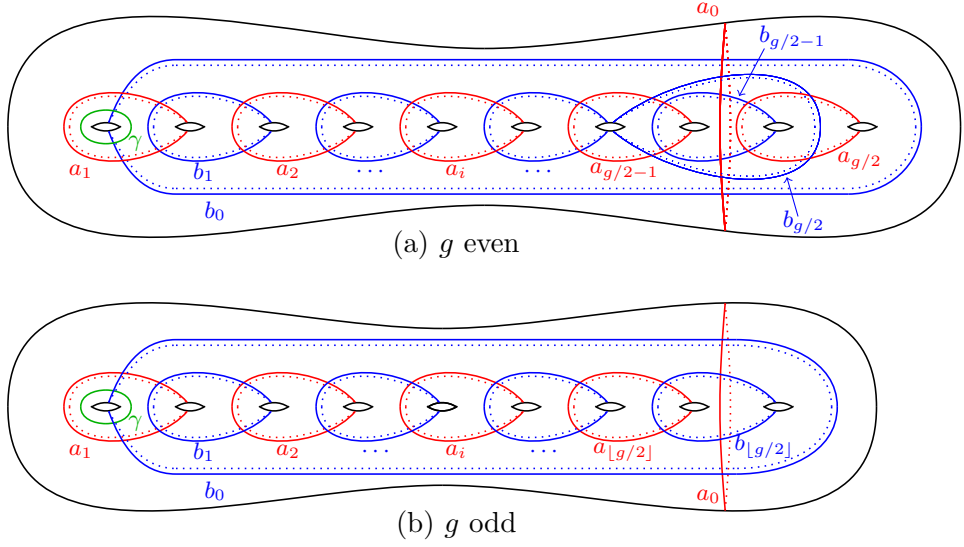
\begin{figure}[t]
\centering
\begin{tikzpicture}[scale=.745]		
		\foreach \y in {2} {
        % surfaces
        \draw[semithick] (2,.9+\y) to [out=90,in=180] (10.5,2.3+\y);
        \draw[semithick] (19,.9+\y) to [out=90,in=0] (10.5,2.3+\y);
        \draw[semithick] (2,.9+\y) to [out=-90,in=180] (10.5,-.5+\y);
        \draw[semithick] (19,.9+\y) to [out=-90,in=0] (10.5,-.5+\y);
        \draw[semithick] (2,.9-3) to [out=90,in=180] (9.75,2.3-3);
        \draw[semithick] (17.5,.9-3) to [out=90,in=0] (9.75,2.3-3);
        \draw[semithick] (2,.9-3) to [out=-90,in=180] (9.75,-.5-3);
        \draw[semithick] (17.5,.9-3) to [out=-90,in=0] (9.75,-.5-3);
		% genus
        \draw[semithick] (3.48,0.91+\y) .. controls (3.65,0.8+\y) and (3.85,0.8+\y) .. (4.02,0.91+\y);
		\draw[semithick] (3.5,0.9+\y) .. controls (3.65,1+\y) and (3.85,1+\y) .. (4,0.9+\y);		
		\draw[semithick] (5,0.9+\y) .. controls (5.15,1+\y) and (5.35,1+\y) .. (5.5,0.9+\y);
		\draw[semithick] (4.98,0.91+\y) .. controls (5.15,0.8+\y) and (5.35,0.8+\y) .. (5.52,0.91+\y);
		\draw[semithick] (6.5,0.9+\y) .. controls (6.65,1+\y) and (6.85,1+\y) .. (7,0.9+\y);
		\draw[semithick] (6.48,0.91+\y) .. controls (6.65,0.8+\y) and (6.85,0.8+\y) .. (7.02,0.91+\y);
        \draw[semithick] (8,0.9+\y) .. controls (8.15,1+\y) and (8.35,1+\y) .. (8.5,0.9+\y);
		\draw[semithick] (7.98,0.91+\y) .. controls (8.15,0.8+\y) and (8.35,0.8+\y) .. (8.52,0.91+\y);
		\foreach \x in {1.5,3,4.5,6,7.5,9}{
		\draw[semithick] (8+\x,0.9+\y) .. controls (8.15+\x,1+\y) and (8.35+\x,1+\y) .. (8.5+\x,0.9+\y);
		\draw[semithick] (7.98+\x,0.91+\y) .. controls (8.15+\x,0.8+\y) and (8.35+\x,0.8+\y) .. (8.52+\x,0.91+\y);
		}
        \foreach \x in {3}{
		\draw[semithick] (3.48,0.91-\x) .. controls (3.65,0.8-\x) and (3.85,0.8-\x) .. (4.02,0.91-\x);
		\draw[semithick] (3.5,0.9-\x) .. controls (3.65,1-\x) and (3.85,1-\x) .. (4,0.9-\x);		
		\draw[semithick] (5,0.9-\x) .. controls (5.15,1-\x) and (5.35,1-\x) .. (5.5,0.9-\x);
		\draw[semithick] (4.98,0.91-\x) .. controls (5.15,0.8-\x) and (5.35,0.8-\x) .. (5.52,0.91-\x);
		\draw[semithick] (6.5,0.9-\x) .. controls (6.65,1-\x) and (6.85,1-\x) .. (7,0.9-\x);
		\draw[semithick] (6.48,0.91-\x) .. controls (6.65,0.8-\x) and (6.85,0.8-\x) .. (7.02,0.91-\x);
        \draw[semithick] (8,0.9-\x) .. controls (8.15,1-\x) and (8.35,1-\x) .. (8.5,0.9-\x);
		\draw[semithick] (7.98,0.91-\x) .. controls (8.15,0.8-\x) and (8.35,0.8-\x) .. (8.52,0.91-\x);
        }
        \foreach \x in {1.5,3,4.5,6,7.5}{
        \draw[semithick] (8+\x,0.9-3) .. controls (8.15+\x,1-3) and (8.35+\x,1-3) .. (8.5+\x,0.9-3);
		\draw[semithick] (7.98+\x,0.91-3) .. controls (8.15+\x,0.8-3) and (8.35+\x,0.8-3) .. (8.52+\x,0.91-3);
        }
        \foreach \x in {3}{
        \draw[semithick] (6.5+\x,0.9-\x) .. controls (6.65+\x,1-\x) and (6.85+\x,1-\x) .. (7+\x,0.9-\x);
		\draw[semithick] (6.48+\x,0.91-\x) .. controls (6.65+\x,0.8-\x) and (6.85+\x,0.8-\x) .. (7.02+\x,0.91-\x);
        }
        %
        %
        % curves for g=even
        %1
        \draw[red] (14.5,3+\y) node {\footnotesize $a_{0}$};
        \draw[red] (3.3,0.1+\y) node {\footnotesize $a_1$};
        \draw[red] (7,0.1+\y) node {\footnotesize $a_2$};
        \draw[red] (10,0.1+\y) node {\footnotesize $a_i$};
        \draw[red] (13,.1+\y) node {\footnotesize $a_{g/2-1}$};
        \draw[red] (17.2,.3+\y) node {\footnotesize $a_{g/2}$};
        \draw[blue] (5.7,-.6+\y) node {\footnotesize $b_{0}$};
        \draw[blue] (5.5,0.1+\y) node {\footnotesize $b_1$};
        \draw[blue] (8.5,0.1+\y) node {\footnotesize $\cdots$};
        \draw[blue] (11.5,0.1+\y) node {\footnotesize $\cdots$};
        \draw[blue] (16.2,-.8+\y) node {\footnotesize $b_{g/2}$};
        \draw[->, blue] (16.1,-.7+\y) to  (15.9,\y);
        \draw[blue] (16,2.5+\y) node {\footnotesize $b_{g/2-1}$};
        \draw[->,blue] (15.5, 2.3+\y) to (15.05, 1.45+\y);
        %\draw[blue] (5.7,2.5+\y) node {\footnotesize $b_{n}$};
        %2
        \draw[red] (8.5+1.5,-2.9) node {\footnotesize $a_i$};
        \draw[red] (7,.7-3.6) node {\footnotesize $a_2$};
        \draw[blue] (5.5,.7-3.6) node {\footnotesize $b_{1}$};
        \draw[red] (3.3,0.7-3.6) node {\footnotesize $a_{1}$};
        \draw[blue] (5.7,0.1-3.7) node {\footnotesize $b_{0}$};
        \draw[red] (14.5,-.3-3.4) node {\footnotesize $a_{0}$};
        \draw[red] (13,0.7-3.6) node {\footnotesize $a_{\lfloor g/2 \rfloor}$};
        \draw[blue] (15.5,0.8-3.6) node {\footnotesize $b_{\lfloor g/2 \rfloor}$};
        \draw[blue] (8.6,-2.9) node {\footnotesize $\cdots$};
        \draw[blue] (11.6,-2.9) node {\footnotesize $\cdots$};
        % b_i
        \foreach \x in {0,3,6,9}{
        \draw[semithick,blue] (3+1.5+\x,.9+\y) to [out=90,in=130] (5.2+1.5+\x,.98+\y);
        \draw[semithick,blue] (3+1.5+\x,.9+\y) to [out=-90,in=-130] (5.2+1.5+\x,.82+\y);
        \draw[semithick,blue,dotted] (3.1+1.5+\x,.9+\y) to [out=90,in=140] (5.2+1.5+\x,.98+\y);
        \draw[semithick,blue,dotted] (3.1+1.5+\x,.9+\y) to [out=-90,in=-140] (5.2+1.5+\x,.82+\y);
        }
        
        %b_0
        	%top
        \draw[semithick,blue] (12.3-3+9,.9+\y) to [out=90,in=0] (8+9,1.7+.2+.2+\y);
        \draw[semithick,blue] (9-5.2,.98+\y) to [out=70,in=-180] (5,1.7+.2+.2+\y);
        \draw[semithick,blue] (5,1.7+.2+.2+\y) to (8+9,1.7+.2+.2+\y);
        \draw[semithick,blue,dotted] (12.2-3+9,.9+\y) to [out=90,in=0] (8+9,1.6+.2+.2+\y);
        \draw[semithick,blue,dotted] (9-5.2,.98+\y) to [out=60,in=-180] (5,1.6+.2+.2+\y);
        \draw[semithick,blue,dotted] (5,1.6+.2+.2+\y) to (8+9,1.6+.2+.2+\y);
        	%bottom
        \draw[semithick,blue] (12.3-3+9,.9+\y) to [out=-90,in=0] (8+9,.1-.2+\y-.2);
        \draw[semithick,blue] (9-5.2,.82+\y) to [out=-70,in=-180] (5,.1-.2+\y-.2);
        \draw[semithick,blue] (5,.1-.2+\y-.2) to (8+9,.1-.2+\y-.2);
        \draw[semithick,blue,dotted] (12.2-3+9,.9+\y) to [out=-90,in=0] (8+9,.2-.2+\y-.2);
        \draw[semithick,blue,dotted] (9-5.2,.82+\y) to [out=-60,in=-180] (5,.2-.2+\y-.2);
        \draw[semithick,blue,dotted] (5,.2-.2+\y-.2) to (8+9,.2-.2+\y-.2);
        %a_i
        \foreach \x in {0,3,6,9,12}{
        \draw[semithick,red] (3+\x,.9+\y) to [out=90,in=130] (5.2+\x,.98+\y);
        \draw[semithick,red] (3+\x,.9+\y) to [out=-90,in=-130] (5.2+\x,.82+\y);
        \draw[semithick,red,dotted] (3.1+\x,.9+\y) to [out=90,in=140] (5.2+\x,.98+\y);
        \draw[semithick,red,dotted] (3.1+\x,.9+\y) to [out=-90,in=-140] (5.2+\x,.82+\y);
        %a_0
        \draw[semithick,red] (14.8,2.75+\y) to [out= -95,in=95] (14.8,-.95+\y);
        \draw[semithick,red,dotted] (14.8,2.75+\y) to [out= -85,in=85] (14.8,-.95+\y);
        % b_g/2
        \draw[blue] (15+1.5,.9+\y) to [out=90,in=40] (18-5.2,.98+\y);
        \draw[blue] (15+1.5,.9+\y) to [out=-90,in=-40] (18-5.2,.82+\y);
        \draw[blue,dotted] (18-1.6,.9+\y) to [out=90,in=35] (18-5.2,.98+\y);
        \draw[blue,dotted] (18-1.6,.9+\y) to [out=-90,in=-35] (18-5.2,.82+\y);
        }
        
        %\gamma
        %\draw[semithick,black!30!green] (2.2,.9-3) to [out=90,in=180] (5.9+.75,2.1-3);
        %\draw[semithick,black!30!green] (9.8+1.5,.9-3) to [out=90,in=0] (5.9+1.5-.75,2.1-3);
        %\draw[semithick,black!30!green] (2.2,.9-3) to [out=-90,in=180] (5.9+.75,-.3-3);
        %\draw[semithick,black!30!green] (9.8+1.5,.9-3) to [out=-90,in=0] (5.9+1.5-.75,-.3-3);
        \draw[semithick,black!30!green] (4.28,-2.4) node {\footnotesize $\gamma$};
        \draw[semithick,black!30!green] (3.3,-2.1) to [out=90,in=180] (3.75,-1.8);
        \draw[semithick,black!30!green] (3.3,-2.1) to [out=-90,in=180] (3.75,-2.4);
        \draw[semithick,black!30!green] (4.2,-2.1) to [out=90,in=0] (3.75,-1.8);
        \draw[semithick,black!30!green] (4.2,-2.1) to [out=-90,in=0] (3.75,-2.4);
        \draw[semithick,black!30!green] (4.28,3-2.4+\y) node {\footnotesize $\gamma$};
        \draw[semithick,black!30!green] (3.3,3-2.1+\y) to [out=90,in=180] (3.75,3-1.8+\y);
        \draw[semithick,black!30!green] (3.3,3-2.1+\y) to [out=-90,in=180] (3.75,3-2.4+\y);
        \draw[semithick,black!30!green] (4.2,3-2.1+\y) to [out=90,in=0] (3.75,3-1.8+\y);
        \draw[semithick,black!30!green] (4.2,3-2.1+\y) to [out=-90,in=0] (3.75,3-2.4+\y);
        %
        %
        % curves for g=odd
        %a_1
       % \draw[semithick,red] (12-3+1.5,.9-3) to [out=90,in=40] (12-5.2+1.5,.98-3);
        %\draw[semithick,red] (12-3+1.5,.9-3) to [out=-90,in=-40] (12-5.2+1.5,.82-3);
        %\draw[semithick,red,dotted] (12-3.1+1.5,.9-3) to [out=90,in=30] (12-5.2+1.5,.98-3);
        %\draw[semithick,red,dotted] (12-3.1+1.5,.9-3) to [out=-90,in=-30] (12-5.2+1.5,.82-3);
        %a_1
        \draw[semithick,red] (3,.9-3) to [out=90,in=130] (5.2,.98-3);
        \draw[semithick,red] (3,.9-3) to [out=-90,in=-130] (5.2,.82-3);
        \draw[semithick,red,dotted] (3.1,.9-3) to [out=90,in=140] (5.2,.98-3);
        \draw[semithick,red,dotted] (3.1,.9-3) to [out=-90,in=-140] (5.2,.82-3);
        %a_2
        \draw[semithick,red] (3+3,.9-3) to [out=90,in=130] (5.2+3,.98-3);
        \draw[semithick,red] (3+3,.9-3) to [out=-90,in=-130] (5.2+3,.82-3);
        \draw[semithick,red,dotted] (6.1,.9-3) to [out=90,in=140] (8.2,.98-3);
        \draw[semithick,red,dotted] (6.1,.9-3) to [out=-90,in=-140] (8.2,.82-3);
        %a_i
        \foreach \x in {3,6}{
        \draw[semithick,red] (3+3+\x,.9-3) to [out=90,in=130] (5.2+3+\x,.98-3);
        \draw[semithick,red] (3+3+\x,.9-3) to [out=-90,in=-130] (5.2+3+\x,.82-3);
        \draw[semithick,red,dotted] (6.1+\x,.9-3) to [out=90,in=140] (8.2+\x,.98-3);
        \draw[semithick,red,dotted] (6.1+\x,.9-3) to [out=-90,in=-140] (8.2+\x,.82-3);
        }
        %a_0
        \draw[semithick,red] (14.8,3.15-3.4) to [out= -95,in=95] (14.8,-.55-3.4);
        \draw[semithick,red,dotted] (14.8,3.15-3.4) to [out= -85,in=85] (14.8,-.55-3.4);
        %b_i
        \foreach \x in {0,3,6,9} { 
        \draw[semithick,blue] (3+1.5+\x,.9-3) to [out=90,in=130] (5.2+1.5+\x,.98-3);
        \draw[semithick,blue] (3+1.5+\x,.9-3) to [out=-90,in=-130] (5.2+1.5+\x,.82-3);
        \draw[semithick,blue,dotted] (3.1+1.5+\x,.9-3) to [out=90,in=140] (5.2+1.5+\x,.98-3);
        \draw[semithick,blue,dotted] (3.1+1.5+\x,.9-3) to [out=-90,in=-140] (5.2+1.5+\x,.82-3);
        }
        %b_0
        	%top
        \draw[semithick,blue] (13.8-3+6,.9-3) to [out=90,in=0] (9+6,1.7+.2-3+.2);
        \draw[semithick,blue] (9-5.2,.98-3) to [out=70,in=-180] (5,1.7+.2-3+.2);
        \draw[semithick,blue] (5,1.7+.2-3+.2) to (9+6,1.7+.2-3+.2);
        \draw[semithick,blue,dotted] (13.7-3+6,.9-3) to [out=90,in=0] (9+6,1.6+.2-3+.2);
        \draw[semithick,blue,dotted] (9-5.2,.98-3) to [out=60,in=-180] (5,1.6+.2-3+.2);
        \draw[semithick,blue,dotted] (5,1.6+.2-3+.2) to (9+6,1.6+.2-3+.2);
        	%bottom
        \draw[semithick,blue] (13.8-3+6,.9-3) to [out=-90,in=0] (9+6,.1-.2-3-.2);
        \draw[semithick,blue] (9-5.2,.82-3) to [out=-70,in=-180] (5,.1-.2-3-.2);
        \draw[semithick,blue] (5,.1-.2-3-.2) to (9+6,.1-.2-3-.2);
        \draw[semithick,blue,dotted] (13.7-3+6,.9-3) to [out=-90,in=0] (9+6,.2-.2-3-.2);
        \draw[semithick,blue,dotted] (9-5.2,.82-3) to [out=-60,in=-180] (5,.2-.2-3-.2);
        \draw[semithick,blue,dotted] (5,.2-.2-3-.2) to (9+6,.2-.2-3-.2);
        }
        \draw (10,0.8) node {$\textrm{(a) $g$ even}$};
        \draw (10,-4.2) node {$\textrm{(b) $g$ odd}$};
\end{tikzpicture}
\caption{Multicurves in the closed surface $S_g$ of genus $g$}
\label{figure:Torelli}
\end{figure}

\begin{proof}
It is enough to find a pseudo-Anosov element $f_g$ in $\I_g$ such that
$\el(f_g) \leq 8/(g-12)$ for all large enough $g$.

Let us assume that $g \geq 13$ and let $n= \lfloor g/2 \rfloor$.
Let $A =\{a_0, \cdots, a_n\}$ and $B = \{b_0, \cdots, b_n \}$ be multicurves 
as in Figure \ref{figure:Torelli} 
so that $A \cup B$ fills the surface $S_g$.
Let us define a mapping class $f_g \in \Mod(S_g)$ by
\begin{displaymath}
f_g= T_B^{-1} T_A,
\end{displaymath}
where $T_A = \prod_{i=0}^n T_{a_i}$ and $T_B^{-1} = \prod_{i=0}^n T_{b_i}^{-1}$ are
multi-twists, i.e., products of Dehn twists.
Then $f_g$ is a pseudo-Anosov mapping class since it is arising from Penner's construction
(for Penner's construction, see \cite[Theorem 14.4]{FarbMargalit12} or \cite{Penner88}).
Moreover, since each curve in $A$ and $B$ is a separating curve, $f_g$ lies in the
Torelli group $\I_g$ for each $g$.

We will show that there exists a simple closed curve $\delta$ such that 
$$d_{\C}(\delta, f_g^{\lceil \frac{g}{4} \rceil -3 } (\delta)) \leq 2,$$
and this implies 
$$\el(f_g) \leq \frac{2}{\lceil \frac{g}{4} \rceil - 3} \leq
\frac{8}{g-12}.$$
To show this, we follow the same notation as in the proof of Theorem 6.1 in \cite{GadreTsai11}.
For a finite collection of curves $\{c_j \}_{j=1}^{k}  \subset A \cup B$ such that
$c_1 \cup \cdots \cup c_k$ is connected, 
let $\NN(c_1 \cdots c_k)$ be the regular neighborhood of $c_1 \cup \cdots \cup c_k$.
Consider the curve $\delta = a_i$ and its image under the iteration of $f_g$. One can see that
\begin{align*}
f_g(a_i) &\subset \NN(b_{i-1} a_i b_i)\\
f_g^{2}(a_i) &\subset \NN(b_{i-2} a_{i-1} b_{i-1} a_i b_i a_{i+1} b_{i+1})\\
f_g^{3}(a_i) &\subset \NN(b_{i-3} a_{i-2} \cdots  a_{i+2} b_{i+2})\\
&\vdots\\
f_g^{j}(a_i) &\subset \NN( b_{i-j} a_{i-j+1} \cdots a_{i+j-1} b_{i+j-1})\\
\end{align*}
as long as the inequalities $i - j \geq 1$ and $i+j-1 \leq \lceil g/2
\rceil -2$ are satisfied.  
If we choose $i = \lceil \frac{g}{4} \rceil$ and $j = \lceil
\frac{g}{4} \rceil -3$, then both inequalities are satisfied. 
In this case, both $a_i$ and $f_g^{j}(a_i)$ are disjoint from the simple closed curve
$\gamma$ in Figure \ref{figure:Torelli}.
Hence we have $d_{\C}(a_i, f_g^{\lceil \frac{g}{4} \rceil -3 }(a_i))
\leq 2$, and this completes the proof.

\end{proof}

Finally, We conclude that Theorem \ref{thm:Torelli} follows from Theorem \ref{thm:lowerTorelli} and Theorem \ref{thm:upperTorelli}.

%%%%%%%%%%%%%%%%%%%%%%%%%%%%%%%%%%%%%%%%%%%%%%%%%%%%%%%%%%%%%%%%%%%%%%%%%%%%%%%
%
%	Lower bound of PB
%
%%%%%%%%%%%%%%%%%%%%%%%%%%%%%%%%%%%%%%%%%%%%%%%%%%%%%%%%%%%%%%%%%%%%%%%%%%%%%%%
\medskip
\section{Lower bound for pure braid group}
\label{section:lowerbound}

Now we consider the pure braid group $\PB_n$ as a subgroup of $\Mod(D_n) \simeq B_n$. 
As we mentioned in the introduction, we cannot use the Lemma \ref{lem:Markovpartition}
for $\PB_n$ since it doesn't satisfy the criterion of the lemma 
(in particular, please see the definition of Lefschetz number for a punctured surface
in \cite{Tsai09}).
So to obtain the lower bound for Theorem \ref{thm:PBn}, we need to use another argument to show that there is a uniform constant $q$, 
independent of $n$, such that $M_{\mathcal{R}}^q$ is a positive matrix,
where $M_{\mathcal{R}}$ is the transition matrix for real branches of the invariant train track.

\begin{thm} \label{thm:lowerPBn}
Let $f$ be a pseudo-Anosov element in $\PB_n$. Then we have
$$\el(f) > \frac{1}{158n-168}.$$
\end{thm}

\begin{proof}
When we obtain $\tau$ using the algorithm in \cite{BestvinaHandel95}, 
we can start from a graph $\Gamma$ in $D_n$ as in Figure \ref{figure:BH}.
%that is an embedded copy of a rose $R_n$,
%i.e., a bouquet of $n$ circles (see Figure \ref{figure:BH}). 
The train track $\tau$ fills the surface $D_n$, that is,
the complement $D_n \setminus \tau$ is the union of (topologically) disks,
once-punctured disks, or boundary-parallel annulus.
Therefore each puncture is contained in a distinct ideal polygon, including a monogon and a bigon.
Note that a puncture in a monogon and a bigon is 
a 1-pronged singularity and a regular point, respectively,
of the invariant foliation of the pseudo-Anosov $f$ (see Figure \ref{figure:monogonbigon}).
We will show that the integer $q$ in Proposition \ref{prop:lowerboundcriterion} can be chosen
to be a uniform constant, independent of $n$.

\begin{figure}[t]
\centering
\begin{tikzpicture}[scale=.45]
		% disk
        \foreach \x in {0,15}{
        \draw[semithick] (6+\x,0) to [out=90,in=0] (0+\x,3);
        \draw[semithick] (-6+\x,0) to [out=90,in=-180] (0+\x,3);
        \draw[semithick] (-6+\x,0) to [out=-90,in=-180] (0+\x,-3);
        \draw[semithick] (6+\x,0) to [out=-90,in=0] (0+\x,-3);
        }
        % punctures in D1
        \foreach \x in {1,2,3,4,5}{
        	\draw ({2*\x-6},0) node {$\circ$};
        	\draw[semithick,red] ({2*\x-6},0) circle (.5);
        }
        \foreach \x in {3,4,5}
        	\draw[semithick,red] (0,-2) to [out=0,in=-90] ({2*\x-6},-.5);
        \foreach \x in {1,2}
        	\draw[semithick,red] (0,-2) to [out=180,in=-90] ({2*\x-6},-.5);
        % punctures in D2
        \foreach \x in {1,2,3,4,5}
        	\draw ({2*\x+9},0) node {$\circ$};
        %\foreach \x in {-4,-2,0,2,4}  {
        %1
        \draw[semithick,red] (14.5-4,0) to [out=90, in=180] (15-4,.5);
        \draw[semithick,red] (15-4,0.5) to [out=0, in=90] (15.5-4,0);
        \draw[semithick,red] (15-4,-1) to [out=90,in=-90] (14.5-4,0);
        \draw[semithick,red] (15-4,-1) to [out=90,in=-90] (15.5-4,0);
        %5
        \draw[semithick,red] (14.5+4,0) to [out=-90, in=-180] (15+4,-.5);
        \draw[semithick,red] (15+4,-0.5) to [out=0, in=-90] (15.5+4,0);
        \draw[semithick,red] (15+4,1) to [out=-90,in=90] (14.5+4,0);
        \draw[semithick,red] (15+4,1) to [out=-90,in=90] (15.5+4,0);
        %2
        \draw[semithick,red] (14.5-2,0) to [out=90, in=180] (15-2,.5);
        \draw[semithick,red] (14.5-2,0) to [out=-90, in=180] (15-2,-.5);
        \draw[semithick,red] (15-2,0.5) to [out=0, in=180] (14,0);
        \draw[semithick,red] (15-2,-0.5) to [out=0, in=180] (14,0);
        %4
        \draw[semithick,red] (15.5+2,0) to [out=90, in=0] (15+2,.5);
        \draw[semithick,red] (15.5+2,0) to [out=-90, in=0] (15+2,-.5);
        \draw[semithick,red] (15+2,0.5) to [out=180, in=0] (16,0);
        \draw[semithick,red] (15+2,-0.5) to [out=180, in=0] (16,0);
        %3
        \draw[semithick,red] (14,0) to [out=0, in=-90] (15,1);
        \draw[semithick,red] (15,1) to [out=-90, in=180] (16,0);
        \draw[semithick,red] (15,-1) to [out=90,in=0] (14,0);
        \draw[semithick,red] (15,-1) to [out=90,in=180] (16,0);
        % real branches
        \draw[semithick,red] (15,-1) to [out=-90, in=-90] (11,-1);
        \draw[semithick,red] (15,1) to [out=90, in=90] (19,1);
        \draw (-4.5, 1) node {$\Gamma$};
        \draw (12.5, 1) node {$\tau$};
\end{tikzpicture}
\caption{An embedded graph in $D_n$ and a train track by Bestvina--Handel algorithm}
\label{figure:BH}
\end{figure}
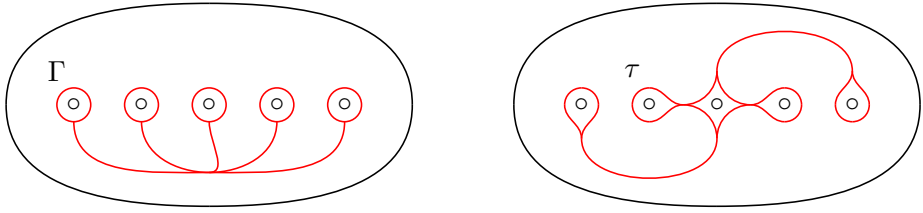

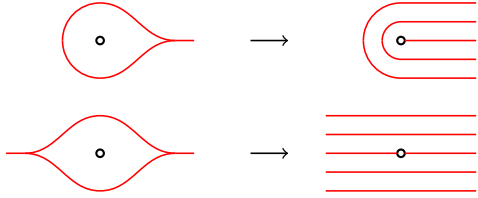
\begin{figure}[t]
\centering
\begin{tikzpicture}[scale=.5]
		% monogon
        \draw[thick] (-4,0) circle [radius=0.1];
        \draw[semithick,red] (-5,0) to [out=90, in=180] (-4,1);
        \draw[semithick,red] (-5,0) to [out=-90, in=180] (-4,-1);
        \draw[semithick,red] (-4,1) to [out=0, in=180] (-2,0);
        \draw[semithick,red] (-4,-1) to [out=0, in=180] (-2,0);
        \draw[semithick,red] (-2,0) to (-1.5,0);
        \draw[semithick,->] (0,0) to (1,0);
        \draw[thick] (4,0) circle [radius=0.1];
        \draw[semithick,red] (3,0) to [out=90, in=180] (4,1);
        \draw[semithick,red] (3,0) to [out=-90, in=180] (4,-1);
        \draw[semithick,red] (3.5,0) to [out=90, in=180] (4,.5);
        \draw[semithick,red] (3.5,0) to [out=-90, in=180] (4,-.5);
        \draw[semithick,red] (4.1,0) to (6,0);
        \draw[semithick,red] (4,0.5) to (6,0.5);
        \draw[semithick,red] (4,-0.5) to (6,-0.5);
        \draw[semithick,red] (4,1) to (6,1);
        \draw[semithick,red] (4,-1) to (6,-1);
        %bigon
        \foreach \x in {3}{
        \draw[thick] (-4,0-\x) circle [radius=0.1];
        \draw[semithick,red] (-6,0-\x) to [out=0, in=180] (-4,1-\x);
        \draw[semithick,red] (-6,0-\x) to [out=0, in=180] (-4,-1-\x);
        \draw[semithick,red] (-4,1-\x) to [out=0, in=180] (-2,0-\x);
        \draw[semithick,red] (-4,-1-\x) to [out=0, in=180] (-2,0-\x);
        \draw[semithick,red] (-2,0-\x) to (-1.5,0-\x);
        \draw[semithick,red] (-6,0-\x) to (-6.5,0-\x);
        \draw[semithick,->] (0,0-\x) to (1,0-\x);
        \draw[thick] (4,0-\x) circle [radius=0.1];
        \draw[semithick,red] (4.1,0-\x) to (6,0-\x);
        \draw[semithick,red] (3.9,0-\x) to (2,0-\x);
        \draw[semithick,red] (2,0.5-\x) to (6,0.5-\x);
        \draw[semithick,red] (2,-0.5-\x) to (6,-0.5-\x);
        \draw[semithick,red] (2,1-\x) to (6,1-\x);
        \draw[semithick,red] (2,-1-\x) to (6,-1-\x);
        }
        \end{tikzpicture}
\caption{Monogon, bigon, and their foliations}
\label{figure:monogonbigon}
\end{figure}

Let $k_1$ be the number of punctures contained in monogons of $\tau$, and let $k_2$ be the number of punctures contained in bigons of $\tau$. 
Also, let $\mathcal{S}$ be the set of singularities of the foliation whose index,
i.e., the number of separatrices, is greater than or equal to 3.
By the Euler--Poincar\'e formula (see \cite{FLP}),
$$k_1 + \sum_{s \in \mathcal{S}} (2-P_s) = 2 \chi(S^2) = 4,$$
where $P_s$ is the index of the singularity $s \in \mathcal{S}$ of the foliation.
Since $P_s \geq 3$, we have
$$k_1 = 4 + \sum_{s \in \mathcal{S}} (P_s-2) \geq 4 + |\mathcal{S}|.$$

On the other hand, since each puncture of $D_n$ either is a 1-pronged singularity, is contained in a bigon, or lies in $\mathcal{S}$, we have 
that $k_1 + k_2  + |\mathcal{S}| \geq n$. Combining the two inequalities, we obtain
$$ k_1 \geq \frac{4+n - k_2}{2}. $$

\textbf{Case 1.} Suppose $k_2 < \frac{1}{2}n$. Then $k_1 > \frac{4+n - n/2}{2} = \frac{1}{4}n + 2$. 
Now we show that there is a monogon in $\tau$ where at most 23 real branches are attached at the cusp of the monogon. 
This implies that the integer $q$ in Proposition \ref{prop:lowerboundcriterion} can be chosen
to be 23.
It is because $f \in \PB_n$ fixes each puncture and hence
for each real branch $\beta$ attached to this monogon,
$f(\beta)$ must passes through one of the other attached real branches.
If there are at most 23 attached real branches, then $f^{23}(\beta)$ must passes through $\beta$.

Suppose on the contrary that each monogon in $\tau$ containing a puncture
has at least 24 real branches attached.
Then there are at least $12k_1 \geq 3n + 24$ real branches.
However, as we explained in Section \ref{section:traintrack}, there are at most $3|\chi(D_n)| = 3n-3$ real branches, which is a contradiction.

\textbf{Case 2.} Suppose $k_2 \geq \frac{1}{2}n$. 
The idea is very similar to case 1. Suppose there are at least 12 real branched attached at the cusps of each bigon of $\tau$ 
containing a puncture. Then there are at least $6k_2 \geq 3n$ real branches in $\tau$ which is again a contradiction. This implies that there is at least one 
bigon with at most 11 real branched attached. Hence by the same logic, one can take $q$ to be $11$. 

In conclusion, for either case 1 or case 2, one can choose $q=23$.
By Proposition \ref{prop:lowerboundcriterion}, we have $w \leq 158n-168$. This finishes the proof. 

\end{proof}

\subsection*{Remark}
We don't expect that the lower bounds of Theorem \ref{thm:lowerTorelli}
and Theorem \ref{thm:lowerPBn} are optimal.
We emphasize that the main purpose of this paper is to describe the asymptotic behavior
of the asymptotic translation length of pseudo-Anosov elements in $\I_g$ and $\PB_n$.

Another quick remark is about the the Euler--Poincar\'e formula
used in the proof above. In \cite{FLP}, the formula is only stated
for compact surfaces. By replacing the puncture with a boundary
component, a $p$-pronged singularity at a puncture is equivalent to 
$p$-many $3$-pronged singularities on the boundary of type
(B) (see section 5.1 of \cite{FLP} for the definition).
Hence the Euler--Poincar\'e formula for a compact surface with boundary
immediately yields the version that we used in this paper.

%%%%%%%%%%%%%%%%%%%%%%%%%%%%%%%%%%%%%%%%%%%%%%%%%%%%%%%%%%%%%%%%%%%%%%%%%%%%%%%
%
%	upper bound of PB
%
%%%%%%%%%%%%%%%%%%%%%%%%%%%%%%%%%%%%%%%%%%%%%%%%%%%%%%%%%%%%%%%%%%%%%%%%%%%%%%%

\medskip
\section{Upper bound of pure braid group} \label{section:upperbound_PB}

Obtaining the upper bound of Theorem \ref{thm:PBn} is completely analogous to 
the proof of Theorem \ref{thm:upperTorelli}.

\begin{thm} 
\label{thm:upperPBn}
For all $n \geq 4$, 
$$L_{\C}(\PB_n) \leq \dfrac{2}{n-3}. $$
\end{thm}

\begin{proof}
It suffices to find a pseudo-Anosov element $f_n$ in $\PB_n$ such that
$\el(f_n) \leq \dfrac{2}{n-3}$.
 Let us assume that $n \geq 4$. Label the punctures
of $D_n$ with $p_1, \cdots, p_n$, 
and for each $i$, consider the simple closed curves
$a_i$ bounding two puncture $p_i$ and $p_{i+1}$ shown in Figure \ref{figure:sscs_Dn}. 

We consider a mapping class $f_n \in \Mod(D_n)$ given by
$$f_n = \prod_{i=1}^{n-1} T_{a_i}^{\epsilon_i},$$
where $\epsilon_i = (-1)^{i+1}$. 
Then $f_n$ is a pseudo-Anosov element since it is arising from Penner's construction.
Moreover, $f_n \in \PB_n$ because each curve $a_i$ is a separating curve.

\begin{figure}[t]
\centering
\begin{tikzpicture}[scale=.5]
		% disk
        \draw[semithick] (6,0) to [out=90,in=0] (0,3);
        \draw[semithick] (-6,0) to [out=90,in=-180] (0,3);
        \draw[semithick] (-6,0) to [out=-90,in=-180] (0,-3);
        \draw[semithick] (6,0) to [out=-90,in=0] (0,-3);
        % punctures
        \foreach \x in {1,2,3}
        	\draw ({2*\x-6.2},0) circle [radius=.1];
        \foreach \x in {4,5}
        	\draw ({2*\x-5.8},0) circle [radius=.1];
        \draw (-4.2,-0.5) node {$p_{1}$};
        \draw (-2.2,-0.5) node {$p_{2}$};
        \draw (4.2,-0.5) node {$p_{n}$};
        \draw (1.1,-0.7) node {$\cdots$};
        %Multicurves
        %red curves
        \foreach \x in {0,1}
        \draw[semithick,red] (4*\x-5,0) to [out=-90,in=-180] (4*\x-3.2,-1.3);
        \foreach \x in {0,1}
        \draw[semithick,red] (4*\x-5,0) to [out=90,in=-180] (4*\x-3.2,1.3);
        \draw[semithick,red] (-1.4,0) to [out=90,in=0] (-3.2,1.3);
        \draw[semithick,red] (-1.4,0) to [out=-90,in=0] (-3.2,-1.3);
        \draw[semithick,red] (4.4-1.2,0) to [out=90,in=0] (4.4-3,1.3);
        \draw[semithick,red] (4.4-1.2,0) to [out=-90,in=0] (4.4-3,-1.3);
        \draw[red] (-3.2,1.8) node {$a_1$};
        \draw[red] (0.5,1.8) node {$a_3$};
        % blue_1
        \draw[semithick,blue] (0.6,0) to [out=90,in=0] (-1.2,1.3);
        \draw[semithick,blue] (-3,0) to [out=90,in=-180] (-1.2,1.3);
        \draw[semithick,blue] (-3,0) to [out=-90,in=-180] (-1.2,-1.3);        
        \draw[semithick,blue] (+0.6,0) to [out=-90,in=0] (-1.2,-1.3);
        % blue_last
        \draw[semithick,blue] (4+1,0) to [out=90,in=0] (4.2-1,1.3);
        \draw[semithick,blue] (4-2.6,0) to [out=90,in=-180] (4.2-1,1.3);
        \draw[semithick,blue] (4-2.6,0) to [out=-90,in=-180] (4.2-1,-1.3);
        \draw[semithick,blue] (4+1,0) to [out=-90,in=0] (4.2-1,-1.3);
        \draw[blue] (-1.2,-1.8) node {$a_2$};
        \draw[blue] (4,-1.7) node {$a_{n-1}$};
        % train track
        %\foreach \x in {0.5}{
        %\draw[semithick] (\x+9.5,2) to (\x+9.5,-2);
        %\draw[semithick] (\x+7.5,0) to (\x+11.5,0);
        %\draw[semithick,->] (\x+12,0) to (\x+13,0);
        %\draw[semithick] (\x+13.5,0) to (\x+17.5,0);
        %\draw[semithick] (\x+15,2) to (\x+15,.5);
        %\draw[semithick] (\x+15,.5) [out=-90, in=180] to (\x+15.5,0) ;
        %\draw[semithick] (\x+16,-2) to (\x+16,-.5);
        %\draw[semithick] (\x+16,-.5) [out=90, in=0] to (\x+15.5,0) ;        
        %\draw (\x+8,0.3) node {$\alpha_{2i+1}$};
        %\draw (\x+10.3,-2) node {$\alpha_{2i}$};
        %\draw (\x+15.3,-2) node {$\tau$};
        %}
        % gamma
        \draw[semithick,black!30!green] (-5.5,0) to [out=90,in=-180] (0,2.5);        
        \draw[semithick,black!30!green] (-5.5,0) to [out=-90,in=-180] (0,-2.5);
        \draw[semithick,black!30!green] (3.5,0) to [out=90,in=0] (0,2.5);
        \draw[semithick,black!30!green] (3.5,0) to [out=-90,in=0] (0,-2.5);
        \draw[black!30!green] (2.3,-2.4) node {$\gamma$};
\end{tikzpicture}
\caption{Simple closed curves in $D_n$}
\label{figure:sscs_Dn}
\end{figure}
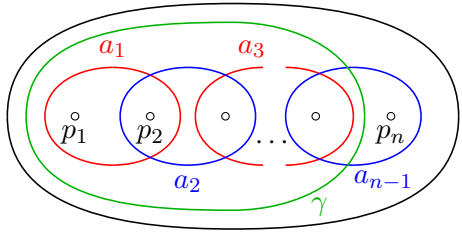

We claim that $d_{\C}(a_1, f_n^{n-1}(a_1)) \leq 2$.
We use the same notation as in the proof of Theorem \ref{thm:upperTorelli}.
Under the iteration of $f_n$, one can see that
\begin{align*}
f_n(a_1) &\subset \NN(a_1a_2)\\
f_n^{2}(a_1) &\subset \NN(a_1a_2a_3)\\
&\vdots\\
f_n^{n-3}(a_1) &\subset \NN(a_1a_2 \cdots a_{n-2}).
\end{align*}
Therefore, we can see that both $a_1$ and $f_n^{n-3}(a_1)$ are disjoint from
the simple closed curve $\gamma$ in Figure \ref{figure:sscs_Dn}.
Therefore, we have
$d_{\C}(a_1, f_n^{n-3}(a_1)) \leq 2$ and
$$\el(f) \leq \frac{2}{n-3}.$$

\end{proof}

Finally, Theorem \ref{thm:PBn} follows from Theorem \ref{thm:lowerPBn} and Theorem \ref{thm:upperPBn}.

%%%%%%%%%%%%%%%%%%%%%%%%%%%%%%%%%%%%%%%%%%%%%%%%%%%%%%%%%%%%%%%%%%%%%%%%%%%%%%%
%
%	Pure mapping class groups
%
%%%%%%%%%%%%%%%%%%%%%%%%%%%%%%%%%%%%%%%%%%%%%%%%%%%%%%%%%%%%%%%%%%%%%%%%%%%%%%%

\medskip
\section{Pure mapping class groups} \label{section:PMod}

In this section, we briefly discuss the case of the pure mapping class
group $\PMod(S_{g,n})$.
Theorem 1.3 of \cite{Valdivia14} showed that the minimal
translation length of the mapping class group $\Mod(S_{g,n})$ behaves
like $1/n$ when $g$ is fixed and $g \geq 2$. The lower bound can be
computed using the same method as in Section
\ref{section:Torelli_lower} and to get the negativity of Lefschetz number,
the assumption that $g \geq 2$ is required. For the upper bound,   
Valdivia constructed an appropriate sequence of pseudo-Anosov
maps in Section 4 of \cite{Valdivia14}. One can observe that 
these are in fact contained in
$\PMod(S_{g,n})$. Hence, Theorem 1.3 of \cite{Valdivia14} and its
proof directly 
imply Theorem \ref{thm:PMod} for $g \geq 2$.

The case $g=0$ is our Theorem \ref{thm:PBn}. For the rest of this
section, we have shown that the computation in Section \ref{section:lowerbound} can be
easily adapted to the case of $\PMod(S_{g,n})$ for any fixed $g \geq 0$.

\begin{thm} \label{thm:lowerPMod} 
Let $g \geq 0$ be fixed. For a pseudo-Anosov element $f \in \PMod(S_{g,n})$, 
$$\ell_{\C}(f) \geq \frac{1}{1296g +638n - 1296}, $$
provided that $n > 38g - 38$.
\end{thm}  

\begin{proof}
Let $f \in \PMod(S_{g,n})$ and let $\tau$ be an invariant train track of
$f$ obtained by Bestvina--Handel algorithm. 
Recall that we only need to compute $q$ so that $M_\mathcal{R}^q$ has
a positive diagonal entry. Here we follow the same notations of the proof of Theorem \ref{thm:lowerPBn}.
The Euler-Poincar\'e formula says 
$$k_1 + \sum_{s \in \mathcal{S}} (2-P_s) = 2 \chi(S_g) = 4-4g.$$
Combining with $k_1 + k_2 + |\mathcal{S}| \geq n$, we
obtain $$ k_1 \geq \frac{n - k_2 + 4 - 4g}{2}.$$

We divide this into two cases as in Section 5. 

\textbf{Case 1.} Suppose $k_2 < \frac{1}{2}n$. Then $k_1 > \frac{1}{4}n + 2 - 2g$. 
Now we show that there is a monogon in $\tau$ where at most 31 real branches are attached at the cusp of the monogon. 
As before, this implies that the integer $q$ in Proposition \ref{prop:lowerboundcriterion} can be chosen
to be 31.

Suppose on the contrary that each monogon in $\tau$ containing a puncture
has at least 32 real branches attached.
Then there are at least $16k_1 \geq 4n + 32-32g$ real branches.
For all $n > 38g-38$, we have $4n + 32-32g > 3n + (38g-38) - (32g-32) = 3n + 6g - 6$. 
This is a contradiction because the number of real branches is at most $3|\chi(S_{g,n})| = 3n + 6g - 6$.

\textbf{Case 2.} Suppose $k_2 \geq \frac{1}{2}n$. 

Suppose further that each of $k_2$ bigons has at least 16 real branches attached. 
Then there are at least $8k_2 \geq 4n$ real branches in $\tau$, and $4n > 3n + 6g - 6$ for all $n > 6g - 6$. 
This implies that there is at least one bigon with at most 15 real branched attached. Hence by the same logic, one can take $q$ to be $15$. 

In conclusion, for either case 1 or case 2, one can choose $q=31$.
By Proposition \ref{prop:lowerboundcriterion}, we have $w \leq 1296g + 638n - 1296$. 
\end{proof}

When $g > 1$, of
course one gets a better lower bound just using the Lefschetz number
argument as in \cite{Valdivia14}, but 
the above proof can be adopted for all $g \geq 0$.

%%%%%%%%%%%%%%%%%%%%%%%%%%%%%%%%%%%%%%%%%%%%%%%%%%%%%%%%%%%%%%%%%%%%%%%%%%%%%%%
%
%	References
%
%%%%%%%%%%%%%%%%%%%%%%%%%%%%%%%%%%%%%%%%%%%%%%%%%%%%%%%%%%%%%%%%%%%%%%%%%%%%%%%

\medskip
\bibliographystyle{alpha} 
\bibliography{purebraid}

\begin{thebibliography}{GHKL13}

\bibitem[AT17]{AougabTaylor15}
T.~{Aougab} and S.~J. {Taylor}.
\newblock {Pseudo-Anosovs optimizing the ratio of Teichm\"uller to curve graph
  translation length}.
\newblock {\em In the tradition of Ahlfors-Bers, VII, Contemporary
  Mathematics}, 696:17--28, 2017.

\bibitem[BB05]{BirmanBrendle05}
Joan~S. Birman and Tara~E. Brendle.
\newblock Braids: a survey.
\newblock In {\em Handbook of knot theory}, pages 19--103. Elsevier B. V.,
  Amsterdam, 2005.

\bibitem[BH95]{BestvinaHandel95}
M.~Bestvina and M.~Handel.
\newblock Train-tracks for surface homeomorphisms.
\newblock {\em Topology}, 34(1):109--140, 1995.

\bibitem[Bow08]{Bowditch08}
Brian~H. Bowditch.
\newblock Tight geodesics in the curve complex.
\newblock {\em Invent. Math.}, 171(2):281--300, 2008.

\bibitem[BSW18]{BaikShinWu18}
Hyungryul Baik, Hyunshik Shin, and Chenxi Wu.
\newblock {Upper bound on translation length on curve graph and fibered face}.
\newblock {\em ArXiv e-prints}, 2018.

\bibitem[BT82]{BottTu82}
Raoul Bott and Loring~W. Tu.
\newblock {\em Differential forms in algebraic topology}, volume~82 of {\em
  Graduate Texts in Mathematics}.
\newblock Springer-Verlag, New York-Berlin, 1982.

\bibitem[FLM08]{FarbLeiningerMargalit08}
Benson Farb, Christopher~J. Leininger, and Dan Margalit.
\newblock The lower central series and pseudo-{A}nosov dilatations.
\newblock {\em Amer. J. Math.}, 130(3):799--827, 2008.

\bibitem[FLP79]{FLP}
A.~Fathi, F.~Laudenbach, and V.~Poenaru.
\newblock {\em Travaux de {T}hurston sur les surfaces}, volume~66 of {\em
  Ast\'erisque}.
\newblock Soci\'et\'e Math\'ematique de France, Paris, 1979.
\newblock S{\'e}minaire Orsay, With an English summary.

\bibitem[FM12]{FarbMargalit12}
Benson Farb and Dan Margalit.
\newblock {\em A primer on mapping class groups}, volume~49 of {\em Princeton
  Mathematical Series}.
\newblock Princeton University Press, Princeton, NJ, 2012.

\bibitem[GHKL13]{GadreHironakaKentLeininger13}
V.~Gadre, E.~Hironaka, R.~P. Kent, IV, and C.~J. Leininger.
\newblock Lipschitz constants to curve complexes.
\newblock {\em Math. Res. Lett.}, 20(4):647--656, 2013.

\bibitem[GP74]{GuilleminPollack74}
Victor Guillemin and Alan Pollack.
\newblock {\em Differential topology}.
\newblock Prentice-Hall, Inc., Englewood Cliffs, N.J., 1974.

\bibitem[GT11]{GadreTsai11}
Vaibhav Gadre and Chia-Yen Tsai.
\newblock Minimal pseudo-{A}nosov translation lengths on the complex of curves.
\newblock {\em Geom. Topol.}, 15(3):1297--1312, 2011.

\bibitem[KS17]{KinShin17}
E.~{Kin} and H.~{Shin}.
\newblock {Small asymptotic translation lengths of pseudo-Anosov maps on the
  curve complex}.
\newblock {\em ArXiv e-prints}, to appear in Groups, Geometry, and Dynamics,
  July 2017.

\bibitem[MM99]{MasurMinsky99}
Howard~A. Masur and Yair~N. Minsky.
\newblock Geometry of the complex of curves. {I}. {H}yperbolicity.
\newblock {\em Invent. Math.}, 138(1):103--149, 1999.

\bibitem[Pen88]{Penner88}
Robert~C. Penner.
\newblock A construction of pseudo-{A}nosov homeomorphisms.
\newblock {\em Trans. Amer. Math. Soc.}, 310(1):179--197, 1988.

\bibitem[PH92]{PennerHarer92}
R.~C. Penner and J.~L. Harer.
\newblock {\em Combinatorics of train tracks}, volume 125 of {\em Annals of
  Mathematics Studies}.
\newblock Princeton University Press, Princeton, NJ, 1992.

\bibitem[Tsa09]{Tsai09}
Chia-Yen Tsai.
\newblock The asymptotic behavior of least pseudo-{A}nosov dilatations.
\newblock {\em Geom. Topol.}, 13(4):2253--2278, 2009.

\bibitem[Val14]{Valdivia14}
Aaron~D. Valdivia.
\newblock Asymptotic translation length in the curve complex.
\newblock {\em New York J. Math.}, 20:989--999, 2014.

\bibitem[Val17]{Valdivia17}
Aaron~D. Valdivia.
\newblock Lipschitz constants to curve complexes for punctured surfaces.
\newblock {\em Topology Appl.}, 216:137--145, 2017.

\end{thebibliography}

\end{document}